\documentclass[12pt,a4paper]{amsart}
\usepackage[utf8x]{inputenc}
\usepackage{ucs}
\usepackage[top=1in, bottom=1in, left=1in, right=1in]{geometry}
\usepackage{amsmath}
\usepackage{amsthm}
\usepackage{amsfonts,dsfont}
\usepackage{changes}
\usepackage{amssymb}
\usepackage{mathtools} 
\usepackage{amsthm}
\usepackage{framed}
\usepackage{mathrsfs} 
\usepackage{graphicx}
\usepackage{enumerate}  

\newtheorem{theoA}{Theorem}

\newcommand{\e}{\rm e}

\newtheorem{thm}{Theorem}[section]
\newtheorem{lem}[thm]{Lemma}
\newtheorem{defi}[thm]{Definition}

\newtheorem{coro}[thm]{Corollary}
\newtheorem{prop}[thm]{Proposition}
\newtheorem{rmk}[thm]{Remark}

\def\E{\mathbb{ E}}
\graphicspath{ {./images/} }  
\newcommand{\Gcal}   {{\mathcal G }}

\begin{document}
\date{}
\title[VRJP and fractional moment localization]{A note on recurrence of the Vertex reinforced jump process and fractional moments localization}
\author[A. Collevecchio]{Andrea Collevecchio }
\author[X. Zeng]{Xiaolin Zeng} 

\begin{abstract}
We give a simple proof for recurrence of  vertex reinforced jump process on \(\mathbb{Z}^d\), under strong reinforcement. Moreover, we show how the previous result implies that linearly edge-reinforced random walk on \ \(\mathbb{Z}^d\)  is {recurrent} for strong reinforcement. Finally, we  prove that the \(H^{(2|2)}\) model on \(\mathbb{Z}^d\)  localizes at strong disorder. Even though these results are well-known, we propose a unified approach, {which  also has the advantage to  provide shorter proofs}, and relies on estimating fractional moments, introduced by Aizenman and Molchanov.
\end{abstract}
\maketitle
\section{Introduction}

The Vertex Reinforced Jump Process (VRJP) is a continuous time self-interacting process. It was first studied by Davis and Volkov (\cite{davis2002continuous} and \cite{davis2004vertex}) on \(\mathbb{Z}\). See also \cite{collevecchio2009,articleBS} and \cite{chen2018speed} for a study of VRJP on trees, {and \cite{2018arXiv181006905R} for super-linear VRJP}.  Recent studies {\cite{ST15}} revealed a close relation between VRJP and linearly edge reinforced random walks (ERRW, introduced by Coppersmith and Diaconis \cite{coppersmith1987random}). Moreover, VRJP is also related to a supersymmetric hyperbolic sigma model (introduced by Zirnbauer~\cite{Zirnbauer91}), called the \(H^{(2|2)}\)-model, studied in~\cite{DS10,DSZ06}. The latter is a toy model for the study of Anderson transition. The  paper \cite{STZ15} introduced a random operator which is naturally related to these objects.

In \cite{ST15} and \cite{angel2014localization}, were given two different proofs of the fact that the ERRW are recurrent under strong reinforcement on \(\mathbb{Z}^d\). These were long-standing open problems in the field. We will give yet an alternative {short} proof, using a unifying approach built on ideas  from \cite{aizenman1993localization}. Moreover, we use the fact that ERRW is a time change of VRJP with random  i.i.d. conductance to prove recurrence of ERRW on \(\mathbb{Z}^d\) when the reinforcement is strong enough.

\section{The model} We define VRJP as follows. {Let} \(\mathcal{G}=(V,E,W, \theta)\) be a non-directed locally finite weighted graph, where to each edge \(e \in E\)  is assigned an initial weight \(W_e \ge0\) and to each vertex \(i\) is assigned an initial weight \(\theta_i>0\).  Moreover, we assume that \(\mathcal{G}\) has self-loops, i.e. edges connecting a vertex to itself; each pair of vertices in \(V\) can be joined by at most one edge, and two vertices \(i,j\) are neighbors, denoted by \(i\sim j\), if they are joined by an edge.  {In this paper,} we mainly focus on \(\mathbb{Z}^d\) and its sub-graphs, with general weights, possibly random.  Denote by \(W = (W_e)_{e \in E}\) and \(\theta = (\theta_i)_{i \in V}\).
 For \(i\sim j\) we use the notation \(W_{i,j}\) for the weight on the edge connecting \(i\) and \(j\). If \(i\) and \(j\) are not neighbors, we set \(W_{i,j} = 0\).
VRJP\((W, \theta)\) on \(\Gcal\) is a  continuous time process that takes values on  \(V\). This process is denoted by  \({\bf Y}= (Y_t)_{t\ge 0}\) and  starts at \(Y_0=i_0\in V\), where \(i_0\in V\) is a designated vertex. Conditionally on the past of \({\bf Y}\) up to time \(t\), and conditionally to  \(Y_t=i\), this process jumps at time \(t\) towards \(j\)  at rate
\begin{equation}
\nonumber
W_{i,j}L_j(t)\ \ \ \text{    where   }L_j(t)=\theta_j+\int_0^t \mathds{1}_{Y_u=j}du.
\end{equation}
{In particular, VRJP can only jump among adjacent vertices.}
It is shown in \cite{ST15} that after a  suitable time change,  the VRJP is a mixture of Markov jump processes. Section~\ref{sec-multi-IG-dist} contains a descriptions of the mixing measure and its very useful properties.

Next, we define the Linearly Edge-Reinforced Random Walk (ERRW). Fix a collection \(a=(a_e)_{e \in E}\) of positive real numbers, they are called the initial weights of ERRW. It is a discrete time process, which takes values on \(V\), and at each step jumps between nearest neighbors, updating the weights on the edges as follows. Initially to each edge \(e\) is assigned a weight \(a_e\). Each time the process traverses an edge, the weight of that edge is increased by 1. The probability to traverse a given edge at   a given time is proportional to the weight of that edge at that time. We denote ERRW\((a)\) for such a process.

We use the notation VRJP(\(W\)) to denote VRJP with edge weights \(W\) and vertex weights \(\theta_i \equiv 1\), that is VRJP(\(W,1\)). It turns out that, as a consequence of Theorem~\ref{thmA-kolmogorov-extension}, or Corollary 1 of \cite{STZ15}, the two models VRJP(\(W,\theta\)) and VRJP(\(W^{\theta},1\)) (where \(W^{\theta}_{i,j}=W_{i,j}\theta_i \theta_j\)) behave the same in term of recurrence/transience. Moreover, most of our results on infinite graphs assume some ergodicity of the model w.r.t. \(\mathbb{Z}^d\)-translation, that is, for simplicity, we will always consider constant \(W_{i,j}\equiv W\) and \(\theta_i\equiv \theta\) on \(\mathbb{Z}^d\), in such case we also have {equivalence among} the models VRJP(\(W,\theta\)), VRJP(\(W^{\theta},1\)) and VRJP(\(1,\sqrt{W}\theta\)). In particular, considering VRJP(\(W\)) is almost as general as considering VRJP(\(W,\theta\)).
\begin{thm}\label{LERRW}
\begin{enumerate}
\item Consider a collection of independent positive random variables  \(\widetilde{W} = \{\widetilde{W}_e \colon e \in E\}\). Consider the process  \({\bf Y}\) defined as follows. Conditionally on \(\widetilde{W}\), \({\bf Y}\) is VRJP\((\widetilde{W})\). There exists \(\overline{W}_d \in (0, \infty]\)  such that if
    \[ \sup_{e \in E} \E\left[\widetilde{W}_{e}^{1/4}\right] {<} \overline{W}_d,\]
    then \({\bf Y}\) is recurrent\footnote{By recurrent we mean that the process visit its starting position infinitely often almost surely.}.\\
\item Consider ERRW\((a)\) on \(\mathbb{Z}^d\), with \(d \ge 1\).
There exists \(\overline{a}(d) \in (0, \infty)\) such that ERRW\((a)\) satisfying \(\sup_{e \in \E} a_e \le \overline{a}(d)\) is recurrent.
\end{enumerate}
\end{thm} 
\begin{rmk} {In our proof, we provide bounds for $\overline{W}_d$. More precisely, }
\begin{equation}\label{eq:lowb}
{\overline{W}_d  \ge \frac{\sqrt{\pi}}{\Gamma(1/4) 2^{5/3}d}\approx \frac{0.24}{d}}.
\end{equation}
In terms of ERRW(\(a\)), \(\overline{W}_d\ge \frac{\Gamma(1/4+a)}{\Gamma(a)}\), which implies, for example, \(\overline{a}(3)\lesssim 0.65\).
\end{rmk}
\begin{coro}
\label{thm-rec-in-str-reinfor}
Consider VRJP\((W)\) on   \(\mathbb{Z}^d\), where  \(W = (W_e)_{e \in E}\) is  a collection of deterministic weights and \(E\) is the edge set   of \(\mathbb{Z}^d\).  Let
  \(\overline{W}_d\)   be as in Theorem~\ref{LERRW}. If \(\sup_{e \in E} {W_e}<\overline{W}_d^4\), then the VRJP\((W)\) is a.s. recurrent.
\end{coro}
Moreover, we were able to apply our proof to show localization of a random Schrödinger operator \(H(\theta)\) (c.f. definition in Theorem~\ref{thm-induction-multiIG-IG}) connected both to the   \(H^{(2|2)}\) model (introduced in \cite{Zirnbauer91}) and  VRJP\((1, \theta)\). We recall that a  random operator is called localized if it has a.s. a complete set of orthonormal eigenfunctions, which decay exponentially (in particular, it has only pure point spectrum). {It is shown in Theorem 1 of \cite{DS10} that, the Green function of \(H(\theta)\) at energy level 0 (ground state) decay exponentially when \(\theta\) is small enough}. Our approach provides an alternative proof of the above in Section~\ref{sec-pp-spec}, see Theorem~\ref{loc}.

The first part of  Theorem~\ref{LERRW}  clearly implies  Corollary~\ref{thm-rec-in-str-reinfor}.  Moreover, Part 2) of Theorem~\ref{LERRW} is a corollary of part 1). In fact, it relies on  a result of  Sabot and Tarr\`es which can be described as follows. If \(\bf Y\) is the process defined in Theorem~\ref{LERRW} part 1), where \(\widetilde{W} =(\widetilde{W}_e)_{e \in E} \) are independent random variables and Gamma(\(a_e\),1) distributed; then the skeleton of \(\bf Y\) (that is, the discrete time process associated to \(\bf Y\)) equals in distribution to ERRW\((a)\) (proof can be found in Section~\ref{sec-vrjp-rwre}). {Hence, we only need to prove Theorem~\ref{LERRW} part 1)}.

The rest of the paper is organized as follows: In Section~\ref{sec-multi-IG-dist} we will introduce a family of mixing measures connected with the local times of VRJP and derive some of its properties for later use. In Section~\ref{sec-vrjp-rwre}, as a preparation, we recall the fact that VRJP is a mixture of Markov jump processes, and to be self-contained, we provide proofs. Section~\ref{sec-rec} is devoted to the proof of recurrence in strong reinforcement, i.e. Theorem~\ref{LERRW} part 1). Finally, in Section~\ref{sec-pp-spec}, we show that, the operator related to the VRJP is localized in strong disorder, as an application of \cite{aizenman1993localization}.
\section{The multivariate inverse Gaussian distribution}
\label{sec-multi-IG-dist}
The aim of this section is to introduce and study the properties of a particular random potential of some Schrödinger operators on finite weighted graphs. This operator is then extended to infinite graphs, in particular \(\mathbb{Z}^d\).

Consider a finite weighted graph \(\mathcal{G}=(V,E, W, \theta)\). For notational reason we identify \(V\) with the set  \(\{1 ,\ldots,N\}\). Recall that to each unordered pair of vertices   \(\{i,j\}\) we assign  a  non-negative weight \(W_{i,j}\), which is strictly positive if and only if \(i\sim j\), and denote \(W = (W_{e})_{e \in E}\) for short. Moreover, to each vertex \(i\in V\) is assigned a real number \(\theta_i>0\), and \(\theta = (\theta_i)_{i \in V}\).

\begin{defi}\label{RSO}  A Schrödinger operator  \(H_\beta\)  on the finite weighted graph   \(\mathcal{G}\) with potential \(2\beta \in \mathbb{R}^N\)  is an \(N\times N\) matrix  with the following entries (\(1\le i,j\le N\))
\begin{equation}
\label{eq-def-H}
H_\beta(i,j)=\begin{cases}  2\beta_i & i=j\\ -W_{i,j} & i\ne j\end{cases}.
\end{equation}
\end{defi}
For any  \(\beta\in \mathbb{R}^N\), let \([\beta]\) be the diagonal matrix where the \(i\)-entry of the diagonal is \(\beta_i\), for \(1 \le i \le N\).
If we denote by  \(\Delta_W\) the weighted graph Laplacian,  whose entries are  \((W_{i,j})_{i,j \in V}\) (recall that \(W_{i, i} =0\)), we have \(H_\beta= 2 [\beta]-\Delta_W\). We call \(2\beta\) the potential, and even though this  choice might differ with part of the literature, we {aim} to be consistent with  the terminology used in few papers that studied VRJP, e.g. in \cite{STZ15,SZ15}.  A main ingredient in the present paper is a random version of \(H_\beta\), with a particular random potential \(\beta\) defined by the following theorem.

The proof of the following theorem is due to Letac  and Jacek \cite{letac2017multivariate} and  can also be found in \cite{SZ15}. For the sake of completeness, we included a proof in the Appendix.
\begin{theoA}
\label{thm-induction-multiIG-IG}Let \(H=H_\beta\) be defined as in \eqref{eq-def-H}. If \(\theta=(\theta_1 ,\ldots,\theta_N)\) and \(\eta=(\eta_1 ,\ldots,\eta_N)\) are  vectors with  real positive coordinates, then
\begin{equation}
\label{eq-multi-IG}
\int_{H>0} \frac{\prod_i \theta_i}{\sqrt{(\pi/2)^N}}  e^{-\frac{1}{2}\left( \left< \theta,H \theta \right>+\left< \eta,H^{-1} \eta \right>-2\left< \theta,\eta \right> \right)}\frac{1}{\sqrt{\det H}}\prod_{i\in V}d\beta_i=1,
\end{equation}
where \(\left< \cdot,\cdot \right>\) is the usual scalar product of \(\mathbb{R}^N\), \(\left< \theta,H \theta \right>=\sum_{1\le i,j\le N}H(i,j)\theta_i \theta_j\), and \(\{H>0\}\) is the collection of \(\beta\) such that \(H_{\beta}\) is positive definite. In particular, the integrand is a probability density and  defines the distribution of an \(N\)-dimensional random vector \(\beta\).
\end{theoA}
The density appearing in the integrand in \eqref{eq-multi-IG} is a multidimensional version  of the  inverse Gaussian distribution. More precisely,  \eqref{eq-multi-IG} is a generalization of  the well-known fact, that for any \(a,b>0\)
\begin{equation}
\label{eq-i-IG-dist}
\frac{a}{\sqrt{\pi/2}}\int_0^\infty e^{-\frac{1}{2}(a^2x+b^2/x-2ab)}\frac{1}{\sqrt{x}}dx=1.
\end{equation}

\begin{rmk}
The density appearing in Theorem~\ref{thm-induction-multiIG-IG} has a rather long history and also have several names. First of all, it is a generalization of  both the Gamma distribution and the Inverse Gaussian {(IG)} distribution. It is related to the magic formula proposed in \cite{coppersmith1987random}, which is then discussed in \cite{diaconis2006bayesian,keane2000edge,merkl2008magic}. In the meantime, it also appears as a hyperbolic supersymmetric measure in \cite{Zirnbauer91,DSZ06,DS10}, which is then identified to the magic formula by Sabot and Tarrès in \cite{ST15}. It is introduced in the above form (with \(\eta=0\)) in \cite{STZ15} and with \(\eta\ne 0\) in \cite{SZ15,DMR15,letac2017multivariate}. 
\end{rmk}

\begin{defi}
\label{defi-nu-G-theta-eta}
We say {that} \(\beta\) is \(\nu^{W,\theta,\eta}\) distributed if its density is defined by \eqref{eq-multi-IG}. Let \(\beta\) be \(\nu^{W,\theta,\eta}\) distributed, and \(H_\beta\) defined as in \eqref{eq-def-H}. Then \(H_\beta\) is a random Schrödinger operator (as its potential is random) associated to the weighted graph \(\mathcal{G}\). Sometimes, we  denote this operator by \(H_{W,\theta,\eta}\) to emphasis its dependency on the parameters.
\end{defi}
For a subset \(V_1\) of \(V\), let us denote \(H_{V_1,V_1}\) the sub-matrix \((H(i,j))_{i,j\in V_1}\) and \(\theta_{V_1}\) is the sub-vector \((\theta_i)_{i\in V_1}\), same for the sub-vector \(\eta_{V_1}\). From the proof of Theorem~\ref{thm-induction-multiIG-IG}, we deduce the following.
\begin{prop}
\label{coro-single-site-decomp}
Let \(\beta\) be a random variable with distribution \(\nu^{W,\theta,\eta}\).  Let \(V_1=V\setminus\{i_0\}\) and \(V_2=\{i_0\}\). The density of the marginal \((\beta_i)_{i\in V_1}\) is
\begin{equation}
\label{eq-density-betaV1}
\nonumber
\left( \frac{2}{\pi} \right)^{|V_1|/2}e^{-\frac{1}{2}\left( \left< \theta_{V_1}, H_{V_1,V_1}\theta_{V_1} \right>+\left< \hat{\eta}_{V_1},\hat{G}^{V_1} \hat{\eta}_{V_1} \right>-2\left< \theta_{V_1},\hat{\eta}_{V_1} \right> \right)}\frac{\mathds{1}_{H_{V_1,V_1}>0}}{\sqrt{\det H_{V_1,V_1}}}\prod_{i\in V_1}\theta_i d\beta_i,
\end{equation}
where \(\hat{\eta}_{V_1}=\eta_{V_1}+W_{V_1,V \setminus{V_1}}\theta_{V\setminus V_1}\) and \(\hat{G}^{V_1}=(H_{V_1,V_1})^{-1}\);
The conditional density of \(\beta_{i_0}\) given \((\beta_i)_{i\in V_1}\) can be expressed as
\begin{equation}
\label{eq-cond-den-betai0}
\mathds{1}_{\gamma>0}\theta_{i_0}\frac{1}{\sqrt{\pi}} e^{- \theta_{i_0}^2 \gamma-\frac{1}{4 \gamma}\check{\eta}_{i_0}^2+\theta_{i_0}\check{\eta}_{i_0}}\frac{1}{\sqrt{\gamma}}d \gamma,
\end{equation}
where the above density is written with the change of variables
\begin{equation}
\label{eq-gamma-is-2betaminus-etc}
\gamma=\beta_{i_0}-\frac{1}{2}\left< W_{i_0,\cdot},\hat{G}^{V_1}W_{i_0,\cdot} \right>=  \beta_{i_0}-\frac{1}{2}\sum_{j:j\sim i_0} W_{i,j}\frac{G(i_0,j)}{G(i_0,i_0)}=\frac{1}{2G(i_0,i_0)},
\end{equation}
and with the notation
\begin{equation}
\label{eq-checketa-i0}
\check{\eta}_{i_0}=\sum_{j\sim i_0}\sum_{k\in V_1}W_{i_0,j}\hat{G}^{V_1}(j,k)\eta_k+\eta_{i_0},
\end{equation}
in such a way that \(\gamma\) is independent of \((\beta_i)_{i\in V_1}\). In particular, if \(\eta\equiv 0\) then \(\gamma\) is Gamma \((\frac{1}{2},\frac{1}{\theta_{i_0}^2})\) distributed.
\end{prop}
\begin{proof}
This is a corollary of the proof of Theorem \ref{thm-induction-multiIG-IG}, given in the Appendix.
It  is a direct application of the fact that the density factorize into product of \eqref{eq-marginal-of-beta} and \eqref{eq-conditional-density-of-beta}. Equation~(\ref{eq-gamma-is-2betaminus-etc}) follows from definition of Green function, that is, for any \(i\in V\), we have
\begin{equation}
\label{eq-beG=I}
2\beta_i G(i_0,i)-\sum_{j:j\sim i}W_{i,j}G(i_0,j)=\mathds{1}_{i=i_0}.
\end{equation}
\end{proof}
\begin{rmk}
Alternatively, it is possible to prove that \(\gamma\) is  Gamma distributed  via  Laplace transform.  Assume that \(\beta\) is \(\nu^{W,\theta,0}\) distributed. Let \(k>0\), define \((\eta_i)_{i\in V}\) by \(\eta_{i_0}=k\) and \(\eta_i=0\) for \(i\ne i_0\), by \eqref{eq-multi-IG}, we see that the Laplace transform of \(\frac{1}{2 \gamma}\) equals
\[\mathbb{E}(e^{-\frac{1}{2}k^2 H^{-1}(i_0,i_0)})=e^{-\sqrt{k^2 \theta_{i_0}}}.\]
This characterize the distribution of \(\frac{1}{2 \gamma}\), and entails that \(\gamma\) is Gamma distributed.
\end{rmk}
The family of measures \(\nu^{W,\theta,\eta}\) satisfy some useful properties, or more specifically they enjoy some symmetries. For example, we can differentiate with respect to the parameters and obtain useful identities. Identities of this kind are called Ward identities, we list two of them, which will be useful later. For other Ward identities, one can check, for example, \cite{DMR15,disertori2017martingales,DSZ06,SZ15,merkl2016convergence,bauerschmidt2018dynkin}.
Recall that we denote by \(\left< \cdot,\cdot \right>\) the usual inner product for vectors in \(\mathbb{R}^d\). Moreover, for any continuous function \(g \colon \mathbb{R}^d \mapsto \mathbb{R}^+\), we define
\[ \left< g(\beta) \right>_{W, \theta, \eta} :=  \int g(\beta)  \nu^{W, \theta, \eta}(d \beta). \]
\begin{coro}[Ward identities]
Consider a finite  weighted graph \(\mathcal{G}\), let \(\beta\) be \(\mu^{W,\theta,\eta}\) distributed.  For any \(|V|\)-dimensional vector  \(k\), such that \(k_i>0\) for all \(i\in V\), we have
\begin{equation}
\label{eq-laplace-tr-nu-theta-eta}
 \left< e^{-\sum_{i\in V}k_i\beta_i}\right>_{W,\theta,\eta}={\rm e}^{-\left< \eta,\sqrt{k+\theta^2}-\theta \right>-\sum_{i,j \colon i\sim j}W_{i,j}(\sqrt{(\theta_i^2+k_i)(\theta_i^2+k_j)}-\theta_i \theta_j)}\prod_{i\in V}\frac{\theta_i}{\sqrt{\theta_i^2+k_i}}.
\end{equation}
Fix \(l\ne i_0\), let \(\Xi(k, W, \beta)=\frac{G(i_0,l)}{G(i_0,i_0)}\exp\left(\frac{k_{i_0}}{2G(i_0,i_0)}-\sum_{i\in V}k_i \beta_i\right)\), then
\begin{equation}
\label{eq-traj-density-integral}
\Big<\Xi(k, W, \beta)\Big>_{W,\theta,0}=\frac{\prod_{i\ne i_0}\theta_i}{\prod_{i\ne l}\sqrt{k_i+\theta_i^2}}\exp\left( -\sum_{i,j \colon i \sim j}W_{i,j}\left( \sqrt{(k_i+\theta_i^2)(k_j+\theta_j^2)}-\theta_i \theta_j \right) \right).
\end{equation}
\end{coro}
\begin{proof}
See Appendix.
\end{proof}
\section{The VRJP is a random walk in random environment}
\label{sec-vrjp-rwre}
Recall that we use \(\mathbf{Y}= (Y_t)_{t \ge 0}\) to denote VRJP(W) defined in the introduction. In~\cite{ST15}, Sabot and Tarr\`es introduced the following time change to the VRJP:
\[Z_t=Y_{D^{-1}(t)}, \qquad \mbox{with \(t \in [0, \infty)\)},\]
where \(D(s)\) is the following increasing random time change
\begin{equation}
\label{eq-Chgt-time}
D(s)=\sum_{i\in V} (L^2_i(s)-\theta_i^2)=\sum_{i\in V}S_i(s),
\end{equation}
in particular, \(S_i(s)=\int_0^s \mathds{1}_{Z_u=i}du\) is the local time of \({\bf Z} = (Z_t)\). 
This section is devoted to the study of   \(\mathbf{Z}\), which turns out to be a mixture of Markov jump processes. By definition, conditionally on \(\{Z_s,s\le t,\ Z_t=i\}\), at time \(t\), \(\mathbf{Z}\) jumps from \(i\) to \(j\) at rate
\begin{equation}
\label{eq-jump-rate-Z}
\nonumber
W_{i,j}L_j(t)dL_i(t)=W_{i,j}\sqrt{S_j(t)+\theta_j^2}\ d \sqrt{S_i(t)+\theta_i^2}=\frac{1}{2} W_{i,j}\sqrt{\frac{S_j(t)+\theta_j^2}{S_i(t)+\theta_i^2}}dS_i(t).
\end{equation}

The following theorem is first proved in \cite{ST15}, the short proof we provided here is inspired by \cite{sabot2013ray,STZ15} and \cite{zeng2013vertex}, it is written in the context of the measure \(\nu^{W,\theta,\eta}\), to provide a self-contained treatise of recurrence of the ERRW.
\begin{theoA}
Consider a finite weighted graph \(\mathcal{G} = (V, E, W, \theta)\).  Let  \({\bf Z}\) be the time changed VRJP starting from \(i_0 \in V\) defined by \eqref{eq-Chgt-time}.  We have that \({\bf Z}\) is a mixture of Markov jump processes.

More precisely, if \(H=H_{W,\theta,0}\) with \(\theta>0\), is the random operator associated to \(\mathcal{G}\) and introduced in Definition \ref{defi-nu-G-theta-eta}. Denote by \({\bf P}\) the distribution of \(G = H^{-1}\). Define, for any collection of positive real numbers \((g(i_0,j))_{j\in V}\), a Markov jump process \(({\bf X}, P^g)\). For which under the measure \(P^g\) (called the quenched measure) the process \({\bf X}\) has Markov jump rates \(\frac{1}{2}W_{i,j}\frac{g(i_0,j)}{g(i_0,i)}\) from \(i\) to \(j\). Then the annealed law of \({\bf X}\), defined by
\[ \mathbb{P}(\cdot) = \int P^g(\cdot) {\bf P}(d g),\]
 equals to the law of \(\mathbf{Z}\).
\end{theoA}
\begin{proof}
 By (\ref{eq-beG=I}), under the  quenched measure \(P^g\), the process \({\bf X}\) has   sojourn time rate at \(i\in V\) equal to
\[\sum_{j:j\sim i}\frac{1}{2} W_{i,j}\frac{g(i_0,j)}{g(i_0,i)}= \begin{cases}\beta_i & i\ne i_0\\ \beta_{i_0}-\frac{1}{2g(i_0,i_0)} & i=i_0\end{cases}, \]

Given a trajectory
\begin{align*}
  \sigma: X_{[0,t_1)}=i_0,\ X_{[t_1,t_2)}=i_1,\ \cdots ,\ X_{[t_{n-1},t_n)}=i_{n-1},\ X_{[t_n,t)}=i_n,
\end{align*} let \(S_i(t)\) be the local time of \({\bf X}\) up to time \(t\). The quenched trajectory density (see Definition 4 of \cite{zeng2013vertex}) of \({\bf X}\) equals
\[f^X_{\text{quenched}}(\sigma)=\prod_{k=1}^n \left( \frac{1}{2} W_{i_{k-1},i_k} \right) \frac{g(i_0,i_n)}{g(i_0,i_0)}\exp\left( -\sum_{i\in V_1}S_i(t) \beta_i+\frac{S_{i_0}(t)}{2g   (i_0,i_0)}\right).\]
By the Ward identity \eqref{eq-traj-density-integral}, the annealed trajectory density of \({\bf X}\) is
\begin{align*}
  &f^X_{\text{annealed}}(\sigma)\\
  &=\prod_{k=1}^n \left( \frac{1}{2} W_{i_{k-1},i_k} \right) \frac{\prod_{i \colon i\ne i_0}\theta_i}{\prod_{i \colon i\ne i_n}\sqrt{S_i(t)+\theta_i^2}}\exp\left( -\sum_{i,j \colon i\sim j}W_{i,j}\left(\sqrt{(S_i(t)+\theta_i^2)(S_j(t)+\theta_j^2)}-\theta_i \theta_j\right) \right).
\end{align*}
On the other hand, the trajectory density of \({\bf Z}\) equals (we abuse notation and denote by \(S_i(t)\) the local time of \({\bf Z}\))
\begin{align*}
  f^{Z}(\sigma)&=\prod_{k=1}^{n}\left( \frac{1}{2} W_{i_{k-1},i_k}\sqrt{\frac{S_{i_k}(t_{k-1})+\theta_{i_k}^2}{S_{i_{k-1}}(t_{k})+\theta_{i_{k-1}}^2}} \right) \cdot \exp\left( -\int_0^t \sum_{j:j\sim Z_v}\frac{1}{2}W_{Z_v,j}\sqrt{\frac{S_j(v)+\theta_j^2}{S_{Z_v}(v)+\theta_{Z_v}^2}} \right)\\
  &=\prod_{k=1}^n \left( \frac{1}{2} W_{i_{k-1},i_k} \right) \frac{\prod_{i\ne i_0}\theta_i}{\prod_{i\ne i_n}\sqrt{S_i(t)+\theta_i^2}} \cdot \exp\left( -\int_0^t \sum_{j:j\sim Z_v}\frac{1}{2}W_{Z_v,j}\sqrt{\frac{S_j(v)+\theta_j^2}{S_{Z_v}(v)+\theta_{Z_v}^2}} \right).
\end{align*}
Now since
\[\frac{d}{dt}\left( \sum_{i,j \colon i\sim j}W_{i,j}\left(\sqrt{(S_i(t)+\theta_i^2)(S_j(t)+\theta_j^2)}-\theta_i \theta_j\right) \right)=\sum_{j:j\sim Z_t}\frac{1}{2}W_{Z_t,j}\sqrt{\frac{S_j(t)+\theta_j^2}{S_{Z_t}(t)+\theta_{Z_t}^2}},\]
we conclude that \(f^Z(\sigma)=f^X_{\text{annealed}}(\sigma)\), hence the law of \({\bf X}\) and \({\bf Z}\) are equal.
\end{proof}
The following corollary is immediate.
\begin{coro}
\label{coro-vrjp-is-random-conductance}
Fix a finite  weighted graph \(\mathcal{G}\). Let \(H_\beta\) the random operator associated to \(\mathcal{G}\) (as per Definition~\ref{defi-nu-G-theta-eta}) and \(G=H^{-1}_\beta\). The discrete time process \((Y_n)\) associated to the VRJP is a random conductance  model, where the conductance on the edge \(\{i,j\}\in E\) is \(C_{i,j}=W_{i,j}G(i_0,i)G(i_0,j)\).
\end{coro}
\section{Recurrence with strong reinforcement on \(\mathbb{Z}^d\)}
\label{sec-rec}
We can actually define the random potential \(\beta\) on a infinite graph, see Proposition 1 of \cite{SZ15} for the case \(\eta\equiv 0\).
\begin{theoA}
\label{thmA-kolmogorov-extension}
One can extend to \(\nu^{W,\theta,\eta}\) to the entire \(\mathbb{Z}^d\). If \((\beta_i)_{i\in \mathbb{Z}^d}\) is \(\nu^{W,\theta,\eta}\) distributed, its law is characterized by the Laplace transform of its finite dimensional marginals: for any finite collection of vertices \(V\subset \mathbb{Z}^d\),
\begin{equation}
\label{eq-finite-marginal-laplace-zd}
\begin{aligned}
\left< {\rm e}^{-\sum_{i\in V}k_i \beta_i}\right>_{W,\theta,\eta}
={\rm e}^{-\sum_{i\in V}(\eta_i\sqrt{k_i+\theta_i^2}-\theta_i)-\sum_{i\in V,j\in \mathbb{Z}^d,i\sim j}W_{i,j} \left( \sqrt{(\theta_i^2+k_i)(\theta_j^2+k_j)}-\theta_i \theta_j \right)}\prod_{i\in V}\frac{\theta_i}{\sqrt{\theta_i^2+k_i}},
\end{aligned}
\end{equation}
where \((k_i)_{i\in V}\) is such that \(k_i>0\), and conventionally \(k_i=0\) if \(i\notin V\).
\end{theoA}
\begin{proof}
This is a direct application of the Kolmogorov extension theorem, and the fact that by \eqref{eq-laplace-tr-nu-theta-eta}, \(\beta_i\) and \(\beta_j\) are independent if \(i\nsim j\).
\end{proof}
\begin{rmk}
\label{rmk-scaling}
In particular, if \(\beta\) is \(\nu^{W,\theta,0}\) distributed, then \(\theta^2 \beta=(\theta_i^2 \beta_i)_{i\in V}\) is \(\nu^{W^{\theta},1,0}\) distributed, where \(W^{\theta}_{i,j}=W_{i,j}\theta_i \theta_j\). Moreover, in the case where \(W\) and \(\theta\) are constant, if \(\beta\) is \(\nu^{W,1,0}\) distributed, then \(\beta/W\) is \(\nu^{1,\sqrt{W},0}\) distributed.
\end{rmk}

In the sequel, we focus on the graph \(\mathbb{Z}^d\). In order to have homogeneity, we assume that  \(\theta_i\) and \(W_{i,j}\) are constant, equals to \(\theta\) and \(W\) respectively. Define the random operator associated to \(\mathbb{Z}^d\), by Definition~\ref{RSO}, it is of the form \(H_{\beta}=2 [\beta] -\Delta_W\), where \(\beta\) is \(\nu^{W,\theta,0}\) distributed, as in Theorem~\ref{thmA-kolmogorov-extension}. Recall that, \(1/2 \beta_i\) is an IG\(\left( \left( \sum_{j:j\sim i}W_{i,j}\theta_j/\theta_i \right)^{-1},\theta_i^2 \right)\) distributed random variable, in particular, the variance of \(\beta_i\) is (since \(W_{i,j}\) and \(\theta_i\) are constant):
\[\left< \beta_i^2 \right>_{W,\theta,0}-\left< \beta_i \right>_{W,\theta,0}^2=\frac{1}{2 \theta_i^2}+\frac{\sum_{j}W_{i,j}\theta_j}{4 \theta_i^3}=\frac{1+dW}{2 \theta^2}.\]
To conjugate our language to the one  that is usually used in the context of  random Schrödinger operator, we have to normalize \(H_{\beta}\) in such a way that the kinetic energy is associated to the unweighted Laplacian \(\Delta\) (that is, with entry 1 or 0 depending whether an edge is present), so we {would} rather consider \(H_{\beta}/W\) instead, and then the coupling constant is the variance of the potential \(2\beta/W\), that is
\[\frac{2}{\theta^2W^2}+\frac{2d}{\theta^2W}.\]
In particular, either \(\theta\) small or \(W\) small will entail large variance, i.e. strong disorder, and we would expect localization.

For a finite box \(\Lambda\) of \(\mathbb{Z}^d\), consider the finite  weighted graph induced by the wired boundary condition, denoted \(\widetilde{\Lambda}\), and defined as follows. Let \(\Lambda=(V,E, W, \theta)\).  The set of vertices of \(\widetilde{\Lambda}\), denoted by  \(V\cup\{\delta\}\) where \(\delta\) is an additional vertex. The graph \(\widetilde{\Lambda}\) is obtained from  \(\Lambda\) by adding edges connecting \(\delta\) to each of the vertices on the boundary, which in turn is defined as \(\{i \colon \exists j \sim i, j \notin V\}\). Moreover
\begin{equation}
\label{eq-def-wired-W}
W_{\delta,i}=\sum_{j:j\sim i, j\notin V}W_{i,j}
\end{equation}
and \(\theta_{\delta}=\theta\).
\begin{lem}
\label{lem-moment-of-greenfct}
There exists a universal constant \(C\), such that for any finite graph \(\mathcal{G}=(V,E)\), in particular, finite boxes of \(\mathbb{Z}^d\), if \(H=H_{W,\theta,0}\) is the associated operator, with  \(\eta\equiv 0\). Then, for any \(i_0\in V\)
\begin{equation}
\label{eq-bound-on-moment-dia-green}
\nonumber
\mathbb{E}\left[  H^{-1}(i_0,i_0)^{1/4} \right]\le C\sqrt{\theta}.
\end{equation}
Note that this bound is completely independent of \(W\).
\end{lem}
\begin{proof}
{We recall that $H^{-1}(i_0,i_0)>0$}. By Corollary~\ref{coro-single-site-decomp}, \(\frac{1}{2} H^{-1}(i_0,i_0)\) is a Gamma variable with parameter \((\frac{1}{2},\frac{1}{\theta^2})\), Hence 
\[\E\left[\frac{1}{2} H^{-1}(i_0,i_0)^{1/4}\right] =  \frac{\Gamma(1/4)}{2^{1/3}\sqrt{\pi}}\sqrt{\theta}.\] 
\end{proof}
From now on, we denote by  \(x\rightarrow y\) the collection of paths connecting \(x\)  to \(y\)  for any pair of vertices \(x\) and \(y\). Moreover, for any connected set \(A  \subset V\), containing \(x, y\),  denote by  \((x\rightarrow y, A)\) the collection of paths connecting \(x\)  to \(y\) and whose vertices belong to \(A\) only.    For any path \(\sigma\), we denote by \(|\sigma|\) its length.
\begin{prop}[Random walk expansion]
\label{orgeb8276c}
Let \(H_\beta\) be a random Schrödinger operator on a finite graph , and let  \(G\) be  relative Green function, i.e. \(G = H^{-1}_\beta\). Let \(\sigma:x\to y\) be a finite admissible path connecting \(x,y\in V\), that is, \(\sigma=(\sigma_1=x, \sigma_2 ,\ldots,\sigma_{|\sigma|}=y)\), define
\[W_{\sigma}=\prod_{k=1}^{|\sigma|-1}W_{\sigma_k,\sigma_{k+1}},\ \ (2 \beta)_{\sigma}=\prod_{k=1}^{|\sigma|}(2 \beta_{\sigma_k}),\]
then we have
\[G(x,y)=\sum_{\sigma \in x\to y}\frac{W_{\sigma}}{(2 \beta)_{\sigma}}.\]
\end{prop}
\begin{proof}
Let \(D\) be the diagonal matrix  \(2[ \beta]\) (for this notation see the paragraph right below Definition~\ref{RSO}).  We have \(H=(I-W \Delta D^{-1})D\).  Since \(H>0\) a.s., we have, by Perron--Frobenius that  \(\Vert W\Delta D^{-1}\Vert<1\). Therefore, we can write
\[G=D^{-1}\sum_{k=0}^{\infty}(W\Delta D^{-1})^k,\]
which is the random walk expansion.
\end{proof}

\begin{proof}[Proof of Theorem~\ref{LERRW} part 1)]
By Remark~\ref{rmk-scaling}, it suffice to deal with the case \(\theta =1\) and \((\widetilde{W}_e)_{e \in E}\) random and independent.
Set  
$$W\rq{} = \sup_e \E[\widetilde{W}^{1/4}_e] < \overline{W}_d. $$
We will combine the random walk expansion and the fact that diagonal Green function is a Gamma variable. Notice that the collection of paths  \(\sigma:0\to x\), can be decomposed as follows.

\begin{figure}
\label{fig1}
\begin{center}
\includegraphics[width=.9\linewidth]{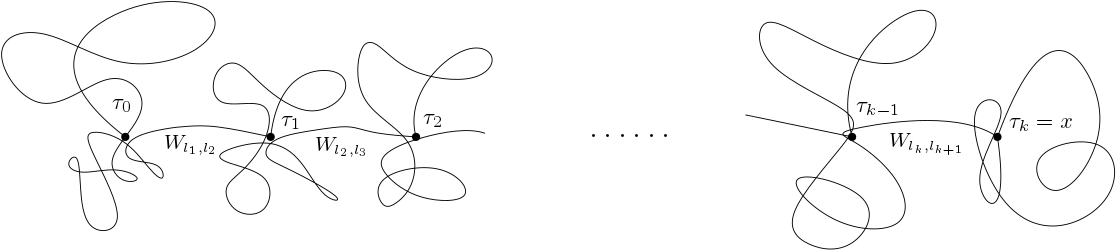}
\caption{A path is decomposed into its loop-erased path along with dangling cycles.}
\end{center}
\end{figure}
As shown in Figure~1, the collection of paths from \(x\) to \(y\), can be seen as the collection of self-avoiding paths from \(x\) to \(y\) (denoted by \(\tau\) in the figure) along with the collection of paths from \(\tau_i\) to \(\tau_i\) for all \(i\) (which are the loops).

We can therefore factorize the random walk expansion according to this path cut, and write
\begin{align*}
&G(0,x)=\sum_{\sigma \in 0\to x}\frac{\widetilde{W}_{\sigma}}{(2 \beta)_{\sigma}}\\
&=\sum_{\tau \in 0\Rightarrow x} \left( \sum_{\pi_0\in 0\to 0} \frac{\widetilde{W}_{\pi_0}}{(2 \beta)_{\pi_0}} \right)\widetilde{W}_{l_1,l_2}\left( \sum_{\pi_1 \in (\tau_1\to \tau_1, V\setminus \{\tau_0\})} \frac{\widetilde{W}_{\pi_1}}{(2 \beta)_{\pi_1}} \right)\widetilde{W}_{l_2,l_3} \cdots \widetilde{W}_{l_k,l_{k+1}}\\
& \left( \sum_{\pi_k \in (\tau_k\to \tau_k, V\setminus\{\tau_0 ,\ldots,\tau_{k-1}\})}\frac{\widetilde{W}_{\pi_k}}{(2 \beta)_{\pi_k}} \right)
\end{align*}
Notice that \(\left( \sum_{\pi_0\in 0\to 0} \frac{\widetilde{W}v_{\pi_1}}{(2 \beta)_{\pi_1}} \right)=G(0,0)\), it is distributed like  the reciprocal of a Gamma random variable, and it is independent of the rest. Unfortunately, the reciprocal of Gamma random variable do not have first moment, thus we compute a fractional moment instead, and use the following bound which holds for any  \(a_1 ,\ldots,a_n>0\), 
\[\left( \sum_i a_i \right)^{1/4}\le \sum_{i} a_i^{1/4}.\]
Hence, if we denote \(C=\mathbb{E}(G(0,0)^{1/4}\), which is a constant, and independent of \((\widetilde{W}_{i,j})\) by Lemma~\ref{lem-moment-of-greenfct}. We have
\begin{align*}
&\mathbb{E}[G(0,x)^{1/4}]\\
&\le \sum_{\tau:0\Rightarrow x} \mathbb{E}\Big[ \Big(G(0,0)\widetilde{W}_{l_1,l_2}  \Big( \sum_{\pi_1:\tau_1\to \tau_1\in V\setminus \{\tau_0\}} \frac{\widetilde{W}_{\pi_1}}{(2 \beta)_{\pi_1}} \Big)\widetilde{W}_{l_2,l_3} \widetilde{W}_{l_k,l_{k+1}} \left( \sum_{\pi_k:\tau_k\to \tau_k\in V\setminus\{\tau_0 ,\ldots,\tau_{k-1}\}}\frac{\widetilde{W}_{\pi_k}}{(2 \beta)_{\pi_k}} \right)\Big)^{1/4}\Big]\\
&= W\rq{}C \sum_{\tau:0\Rightarrow x}\mathbb{E}\left[ \left(\left( \sum_{\pi_1:\tau_1\to \tau_1\in V\setminus \{\tau_0\}} \frac{\widetilde{W}_{\pi_1}}{(2 \beta)_{\pi_1}} \right)\widetilde{W}_{l_2,l_3} \cdots \widetilde{W}_{l_k,l_{k+1}} \left( \sum_{\pi_k:\tau_k\to \tau_k\in V\setminus\{\tau_0 ,\ldots,\tau_{k-1}\}}\frac{\widetilde{W}_{\pi_k}}{(2 \beta)_{\pi_k}} \right)\right)^{1/4} \right]
\end{align*}
Notice that 
\[ \left( \sum_{\pi_1:\tau_1\to \tau_1\in V\setminus \{\tau_0\}} \frac{\widetilde{W}_{\pi_1}}{(2 \beta)_{\pi_1}} \right) \le \sum_{\pi:\tau_1\to \tau_1}\frac{\widetilde{W}_{\pi}}{(2 \beta)_{\pi}}=G(\tau_1,\tau_1)\]
and the rest inside the expectation is independent of \(G(\tau_1,\tau_1)\). Hence continuing our inequality we have
\begin{align*}
\mathbb{E}[G(0,x)^{1/4}]&\le (W\rq{}C)^{2} \sum_{\tau:0\Rightarrow x} \mathbb{E}\left[ \left(\left( \sum_{\pi_2:\tau_2\to \tau_2\in V\setminus \{\tau_0,\tau_1\}} \frac{\widetilde{W}_{\pi_2}}{(2 \beta)_{\pi_2}} \right)\widetilde{W}_{l_3,l_4} \right. \right. \cdots\\
& \left. \left. \widetilde{W}_{l_k,l_{k+1}} \left( \sum_{\pi_k:\tau_k\to \tau_k\in V\setminus\{\tau_0 ,\ldots,\tau_{k-1}\}}\frac{\widetilde{W}_{\pi_k}}{(2 \beta)_{\pi_k}} \right)\right)^{1/4} \right],
\end{align*}
recursively we get, at the end
\begin{align}\label{eq:decay}
\mathbb{E}[G(0,x)^{1/4}]&\le \sum_{\tau:0\Rightarrow x} (W\rq{}C)^{|\tau|}\le \sum_{k\ge |x|}\sum_{\tau:0\Rightarrow x, |\tau|=k} (W\rq{} C)^{k}\le  \sum_{k\ge |x|} (2d)^k (W\rq{}C)^{k}\le e^{-\kappa |x|}
\end{align}
where we have used the fact that there is at most \((2d)^k\) self avoiding paths on a finite sub graph of \(\mathbb{Z}^d\) of length \(k\), and have chosen \(W\rq{}\) {such that \((2d) (W\rq{}C)<1\)}.
\end{proof}
\section{Pure point spectrum}\label{sec-pp-spec}

\begin{thm}\label{loc}
Consider the graph \(\mathbb{Z}^d\) where to each vertex is assigned weight \(\theta>0\), and to each edge weight 1. In this context, denote the operator introduced in Definition~\ref{defi-nu-G-theta-eta} by \(H_\theta\).
There exists \(\theta_0\), which depends on \(d\) only, such that if \(\theta<\theta_0\) we have that the operator \(H_\theta\) is localized, i.e. has a.s. a complete set of orthonormal eigenfunctions, which decay exponentially. 
\end{thm}
\begin{proof}
Since \(\beta= (\beta_i)_{i\in V(\mathbb{Z}^d)}\) be \(\nu^{1,\theta,0}\) distributed, by the discussion after Remark~\ref{rmk-scaling}, \(H_{\theta}\) is in the right scaling. We use a result of  \cite{aizenman1993localization}, in particular, we recall the following definition:
\begin{defi}A probability measure \(\nu\) on \(\mathbb{R}\) is said to be \(\tau\)-regular for some parameter \(\tau\in (0,1)\) if  there exists finite constants \(C, K >0\) such that 
\[ 
\nu([z- \delta, z+\delta]) \le C \delta^\tau \nu([z- K, z+K]),
\]
for all \(\delta \in (0,1)\) and \(z \in \mathbb{R}\).
\end{defi}
According to Theorem 3.1 of \cite{aizenman1993localization}   it suffices for our purpose to show that, the conditional density of single site potential is \(\tau\)-regular for some \(\tau>0\).  Let \(\Gcal_n = (V_n, E_n)\) be the  cube of volume \(2^d n^d\) centered at the origin. 
Denote by \(\partial V_n\) the boundary of this cube, i.e. the set of vertices \(y\notin V_n\) which are neighbors to at least  one element of \(V_n\). By (\ref{eq-finite-marginal-laplace-zd}), the marginal distribution \((\beta_i)_{i\in V}\) is \(\nu^{1,\theta,\eta}\) distributed, where  
\[\eta_i=\sum_{j:j\sim i, j \in \partial V_n}\theta = |\partial V_n| \theta,   \qquad \mbox{ for \(i \in V_n\).}\]
{By Remark~\ref{rmk-scaling} and the proof of Theorem~\ref{LERRW} part 1), we can choose $\theta_0$ such that if we consider VRJP$(1, \theta)$, with $\theta <\theta_0$,} then 
\begin{equation}\label{eq:decay1}
{\mathbb{E}[G(0,x)^{1/4}]\le {\rm e}^{-\kappa x},}
\end{equation}
for some $\kappa >0$.
In the sequel assume \(\theta<\theta_0\).  The conditional density of \(\beta_0\) given \((\beta_i)_{i\in V\setminus\{0\}}\) is given by \eqref{eq-cond-den-betai0}, with \(i_0 = 0\). Using \eqref{eq-checketa-i0}, we have
\begin{align*}
  \check{\eta}_{0}=\sum_{j \colon j\sim 0}\sum_{k}\hat{G}^{V}(j,k) \theta \eta_k
                    \le O(n^{d-1}) e^{-\kappa n}\xrightarrow[]{n\to \infty}0.
\end{align*}
Define \(D_0=\sum_{j \colon j\sim i_0}\frac{G(i_0,j)}{G(i_0,i_0)}\). {Notice that  \( D_0 < \infty\) a.s.}{ In fact, using Cauchy-Schwartz}
\[{\E\left[\left(\frac{G(i_0,j)}{G(i_0,i_0)}\right)^{1/8}\right] \le \E[G(i_0,j)^{1/4}]^{1/2} \E[ G(i_0,i_0)^{-1/4}]^{1/2} < \infty.}\]
{The finiteness of the expression above is a consequence of \eqref{eq:decay1} and the fact that \(\gamma\) (see equation \eqref{eq-gamma-is-2betaminus-etc}) is Gamma distributed.}
 It follows that, by Corollary~\ref{coro-single-site-decomp}, and taking \(V\uparrow \mathbb{Z}^d\), the density of \(\beta_0\) conditioned on \((\beta_i)_{i\in V(\mathbb{Z}^d)\setminus\{0\}}\) equals
\begin{equation}
\label{eq-cond-den-beta0}
g_{D_0}(u):=\frac{1}{\sqrt{2\pi}}\exp\left( -(u-D_0)\theta^2/2 \right)\frac{\theta}{\sqrt{u-D_0}}\mathds{1}_{u\ge D_0}du.
\end{equation}
To check the \(\tau\)-regularity, it suffices to check Equation (3.1) of \cite{aizenman1993localization} at the singularity \(D_0\). We can explicitly compute that,
\begin{equation}
\label{eq-regulartiry-den}
\nonumber
\int_{D_0}^{D_0+x}g_{D_0}(u)du=\operatorname{Erf}(\theta \sqrt{x/2})=\int_0^{\theta\sqrt{x/2}} \frac{1}{\sqrt{2\pi}}e^{-t^2}dt=\frac{2}{\sqrt{\pi}}\left( x-\frac{x^3}{3}+O(x^5) \right).
\end{equation}
Therefore, the conditional single site density (\ref{eq-cond-den-beta0}) is \(1\)-regular (c.f. p256 of \cite{aizenman1993localization}). Thus, by Theorem 3.1 of \cite{aizenman1993localization}, there exists a \(\theta_0\), such that for \(\theta<\theta_0\), the operator \(H_{\theta}\) on \(\mathbb{Z}^d\) is localized.
\end{proof}

\section{Appendix A: Proof of Theorem~\ref{thm-induction-multiIG-IG}}
\begin{proof}
 Recall that \(V=\{1 ,\ldots,N\}\).  We partition \(V\) into \(V=V_1\sqcup V_2\).  After a relabelling,  we can pick \(V_1=\{1 ,\ldots,N'\}\) and \(V_2=\{N'+1 ,\ldots,N\}\).  The matrix \(H\) can be written in block form
\begin{equation}\label{decomposition}
H=\begin{pmatrix}H_{V_1,V_1}& -W_{V_1,V_2}\\ -W_{V_2,V_1} & H_{V_2,V_2}\end{pmatrix}\end{equation}
where the top left block \(H_{V_1,V_1}\) is an \( N'\times N'\) matrix  whose entries are  \((H_{i,j})_{i,j \in V_1}\). The other blocks are defined implicitly by the identity \eqref{decomposition}. The Schur decomposition of the right-hand side of \eqref{decomposition} implies
\[H=\begin{pmatrix}I_{V_1} & 0 \\ -W_{V_2,V_1}\hat{G}^{V_1} & I_{V_2}\end{pmatrix} \begin{pmatrix}H_{V_1,V_1}& 0 \\ 0 & \check{H}^{V_2}\end{pmatrix} \begin{pmatrix}I_{V_1} & -\hat{G}^{V_1}W_{V_1,V_2}\\ 0 & I_{V_2}\end{pmatrix}\]
where \(I_{V_1}\)  (resp.  \(I_{V_2}\) ) is the identity matrix of dimension \(N'\) (resp. \(N- N'\)). Moreover,
\[\hat{G}^{V_1}=(H_{V_1,V_1})^{-1},\ \check{H}^{V_2}=H_{V_2,V_2}-W_{V_2,V_1}\hat{G}^{V_1}W_{V_1,V_2},\ \check{G}^{V_2}=(\check{H}^{V_2})^{-1}.\]
The inverse of \(H\) can be computed via this decomposition, that is
\[G=H^{-1}=\begin{pmatrix}I_{V_1} & \hat{G}^{V_1}W_{V_1,V_2}\\ 0 & I_{V_2}\end{pmatrix}\begin{pmatrix} \hat{G}^{V_1}& 0 \\ 0 & \check{G}^{V_2}\end{pmatrix} \begin{pmatrix}I_{V_1} & 0 \\ W_{V_2,V_1}\hat{G}^{V_1}& I_{V_2}\end{pmatrix}\]
Note we can also write our vector \(\theta,\eta\) in block form, that is, e.g. \(\theta=(\theta_{V_1}, \theta_{V_2})\). Recall that  \(\check{\eta}_{V_2}\) was defined  via  the equation 
\begin{equation}
\label{eq-etacheck}
\begin{pmatrix}I_{V_1}& 0 \\ W_{V_2, V_1}\hat{G}^{V_1}& I_{V_2}\end{pmatrix}\begin{pmatrix}\eta_{V_1}\\ \eta_{V_2}\end{pmatrix}=\begin{pmatrix}\eta_{V_1}\\ \check{\eta}_{V_2}.\end{pmatrix}
\end{equation}
and
\begin{equation}
\label{eq-etahat}
\hat{\eta}_{V_1}=\eta_{V_1}+W_{V_1,V \setminus{V_1}}\theta_{V\setminus V_1}.
\end{equation}

With all these definitions and decompositions, the quadratic form in the exponent (which is the Hamiltonian or energy in the point of view of statistical physics) can be written as
\begin{align*}
\left< \theta,H \theta \right>+\left< \eta,G\eta \right>-2\left< \theta,\eta \right>&=\left< \theta_{V_2},\check{H}^{V_2} \theta_{V_2} \right>+\left< \check{\eta}_{V_2},\check{G}^{V_2}\check{\eta}_{V_2} \right>-2\left< \check{\eta}_{V_2},\theta_{V_2} \right>\\
&+\left< \theta_{V_1},H_{V_1,V_1}\theta_{V_1} \right>+\left< \hat{\eta}_{V_1},\hat{G}^{V_1}\hat{\eta}_{V_1} \right>-2\left< \hat{\eta}_{V_1},\theta_{V_1} \right>.
\end{align*}
Note also that for the determinant, we have, again by Schur, \(\det H=\det H_{V_1,V_1} \det \check{H}^{V_2}\), therefore, the integrand in Equation \eqref{eq-multi-IG} factorizes
\begin{equation}\label{factor}
\begin{aligned}
&\left( \frac{2}{\pi} \right)^{N/2} e^{-\frac{1}{2}\left( \left< \theta,H \theta \right>+\left< \eta,G\eta \right>-2\left< \theta,\eta \right> \right)} \frac{\mathds{1}_{H>0}}{\sqrt{\det H}} \prod_i \theta_i d\beta\\
=&\left( \frac{2}{\pi} \right)^{|V_1|/2} e^{-\frac{1}{2} \left(\left< \theta_{V_1},H_{V_1,V_1}\theta_{V_1} \right>+\left< \hat{\eta}_{V_1},\hat{G}^{V_1}\hat{\eta}_{V_1} \right>-2\left< \hat{\eta}_{V_1},\theta_{V_1} \right>   \right)} \frac{\mathds{1}_{H_{V_1,V_1}>0}}{\sqrt{\det H_{V_1,V_1}}}\prod_{i\in V_1}\theta_i d\beta_{V_1}\\
\cdot& \left( \frac{2}{\pi} \right)^{|V_2|/2} e^{-\frac{1}{2}\left(\left< \theta_{V_2},\check{H}^{V_2} \theta_{V_2} \right>+\left< \check{\eta}_{V_2},\check{G}^{V_2}\check{\eta}_{V_2} \right>-2\left< \check{\eta}_{V_2},\theta_{V_2} \right> \right)}\frac{\mathds{1}_{\check{H}^{V_2}>0}}{\sqrt{\det \check{H}^{V_2}}}\prod_{i\in V_2}\theta_i d\beta_{V_2}
\end{aligned}
\end{equation}
{Next we choose \(V_1 =\{1\}\). Using the decomposition in \eqref{factor}, we have that the first expression in the factorization is of the form
\[\frac{a}{\sqrt{\pi/2}} e^{-\frac{1}{2}(a^2y-b^2/y-2ab)}\frac{1}{\sqrt{y}},\]
for some \(a, b>0\).  In virtue of \eqref{eq-i-IG-dist}, this is a probability density.
By iterating the previous step to each of the one-dimensional marginals, we prove Equation \eqref{eq-multi-IG}. Moreover, the factorization in \eqref{factor} implies that}
\begin{equation}
\label{eq-marginal-of-beta}
\left( \frac{2}{\pi} \right)^{|V_1|/2} e^{-\frac{1}{2} \left(\left< \theta_{V_1},H_{V_1,V_1}\theta_{V_1} \right>+\left< \hat{\eta}_{V_1},\hat{G}^{V_1}\hat{\eta}_{V_1} \right>- 2\left< \hat{\eta}_{V_1},\theta_{V_1} \right> \right)} \frac{\mathds{1}_{H_{V_1,V_1}>0}}{\sqrt{\det H_{V_1,V_1}}}\prod_{i\in V_1}\theta_i d\beta_{V_1}
\end{equation}
is the marginal distribution of \((\beta_i)_{i\in V_1}\), and the second term,
\begin{equation}
\label{eq-conditional-density-of-beta}
\left( \frac{2}{\pi} \right)^{|V_2|/2} e^{-\frac{1}{2}\left(\left< \theta_{V_2},\check{H}^{V_2} \theta_{V_2} \right>+\left< \check{\eta}_{V_2},\check{G}^{V_2}\check{\eta}_{V_2} \right>-2\left< \check{\eta}_{V_2},\theta_{V_2} \right>  \right)}\frac{\mathds{1}_{\check{H}^{V_2}>0}}{\sqrt{\det \check{H}^{V_2}}}\prod_{i\in V_2}\theta_i d\beta_{V_2}
\end{equation}
is the conditional density of \((\beta_i)_{i\in V_2}\) given the values of \((\beta_j)_{j\in V_1}\).
\end{proof}
\section{Appendix B: Proof of Ward identities}
\label{sec-app1}
Let us prove the first Ward identity. Note the \(\nu^{W,\theta,\eta}\) is a probability density for any parameters, in particular true for \(\nu^{W,\sqrt{k+\theta^2},\eta},\), now (\ref{eq-laplace-tr-nu-theta-eta}) is equivalent to the fact that \(\nu^{W,\sqrt{k+\theta^2},\eta}\) is a probability, since
\[\sum_{i\in V}k_i \beta_i+\frac{1}{2}\left< \theta,H \theta \right>=\frac{1}{2}\left< \sqrt{k+\theta^2},H\sqrt{k+\theta^2} \right>+\sum_{i,j}W_{i,j}\left( \sqrt{(k_i+\theta_i^2)(k_j+\theta_j^2)}-\theta_i \theta_j \right).\]

For the second Ward identity, first note that
\begin{equation}
\label{equation-Xi-bis}
\Xi(k,W,\beta)=\frac{G(i_0,l)}{G(i_0,i_0)}\exp\left( -\frac{k_{i_0}}{2}\left< W_{i_0,\cdot},\hat{G} W_{i_0,\cdot} \right>-\frac{1}{2}\left< \sqrt{k},2 \beta \sqrt{k} \right>_{i\ne i_0} \right).
\end{equation}
The factorization with \(V_2=\{i_0\}\) and \(V_1=V\setminus\{i_0\}\) gives
\begin{equation}
\label{equation-factori-i_0-nu-dbeta}
\begin{aligned}
&\nu^{W,\theta,0}(d \beta)=\mathds{1}_{\gamma>0} e^{- \theta_{i_0}^2 \gamma} \frac{\theta_{i_0}}{\sqrt{\pi \gamma}} d \gamma \cdot \\
&\mathds{1}_{\hat{H}>0}\sqrt{\frac{2}{\pi}}^{|V|-1} e^{ -\frac{1}{2}\left< \theta,2 \beta\theta \right>_{i\ne i_0}+\frac{1}{2}\left< \theta, \Delta_W \theta \right>_{i\ne i_0}-\frac{1}{2} \theta_{i_0}^2 \left< W_{i_0,\cdot},\hat{G}W_{i_0,\cdot} \right>+ \theta_{i_0}\left< \theta,W_{i_0,\cdot} \right>_{i\ne i_0} }\frac{1}{\sqrt{\det \hat{H}}}\prod_{i\ne i_0} \theta_i d \beta_i
\end{aligned}
\end{equation}
Therefore, we see that \(\Xi(k,W,\beta)\) plug in well in \(\nu^{W,\theta,0}\). More precisely, we can write (note that \(\frac{G(i_0,l)}{G(i_0,i_0)}=\sum_{k\sim i_0}W_{i_0,k}\hat{G}(k,l)\) which is independent of \(\gamma\))
\begin{equation}
\label{equation-E-Xi}
\begin{aligned}
&\left< \Xi(k,W,\beta) \right>_{W,\theta,0}=\int_{\gamma>0} e^{-\theta_{i_0}^2 \gamma}\frac{\theta_{i_0}}{\sqrt{\pi \gamma}} d \gamma \cdot \int_{\hat{H}>0} \frac{G(i_0,l)}{G(i_0,i_0)} \sqrt{\frac{2}{\pi}}^{|V|-1}\\
& e^{-\frac{1}{2}\left< \sqrt{\theta^2+k}, 2 \beta \sqrt{\theta^2+k} \right>_{i\ne i_0}-\frac{1}{2}(\theta_{i_0}^2+k_{i_0})\left< W_{i_0,\cdot},\hat{G}W_{i_0,\cdot} \right>} \frac{1}{\sqrt{\det \hat{H}}}\prod_{i\ne i_0} \theta_i d \beta_i.
\end{aligned}
\end{equation}
As \(\gamma\) is independent of the rest in \eqref{equation-E-Xi}, using the fact that \(\nu^{W,\sqrt{\theta^2+k},0}\) is a probability measure and
\begin{equation}
\label{equation-ward-3}
\left< \frac{G(i_0,l)}{G(i_0,i_0)} \right>_{W,\sqrt{\theta^2+k},0}=\sqrt{\frac{\theta_l^2+k_l}{\theta_{i_0}^2+k_{i_0}}},
\end{equation}
We deduce the second Ward identity.

Equation~\eqref{equation-ward-3} is well known, it first appear in Equation (B.3) of \cite{DSZ06} (where \(\theta\) and \(k\) are constant), as a consequence of supersymmetric localization. However, to be self-contained, we provide a non supersymmetric proof of it. The idea is to again to use symmetry, note that the second line of Equation~\eqref{equation-factori-i_0-nu-dbeta} can be written as (where the sum \(\sum_{i\sim j}\) is over all non oriented edges with end points in \(V\))
\begin{equation}
\label{equation-marginal-as-GoverG}
\int_{\hat{H}>0} \sqrt{\frac{2}{\pi}}^{|V|-1} e^{-\frac{1}{2}\sum_{i\sim j}W_{i,j}\left( \frac{G(i_0,j)}{G(i_0,i)} \theta_i^2+\frac{G(i_0,i)}{G(i_0,j)}\theta_j^2-2 \theta_i \theta_j \right)} \frac{1}{\sqrt{\det \hat{H}}} \prod_{i\ne i_0} \theta_i d \beta_i=1.
\end{equation}
The above integral, written as integral over the variables \(\psi_i=\frac{G(i_0,i)}{G(i_0,i_0)}\left( \prod_{j\ne i_0}\frac{G(i_0,j)}{G(i_0,i_0)} \right)^{-1/|V|}\), is as follows:
\[\int_{\prod_{i\in V} \psi_i=1} \left(\prod_{i\ne i_0} \theta_i\right) \psi_{i_0} \Theta(\psi_i,\ i\in V) \prod_{i\ne i_0} \frac{d \psi_i}{\psi_i}=1\]
where \(\Theta(\psi_i,\ i\in V) \prod_{i\ne i_0} \frac{d \psi_i}{\psi_i}\) is an expression invariant under the action of \(\mathfrak{S}_{|V|}\) (the symmetric group of permutations) over the simplex \(\prod_{i\in V}\psi_i=1\). The integral equals 1 for any \(W,\ \theta\); thus when integrating \(\frac{G(i_0,l)}{G(i_0,i_0)}=\frac{\psi_l}{\psi_{i_0}}\), it suffice to tinker the symmetric breaking part \(\left( \prod_{i\ne i_0}\theta_i \right) \psi_{i_0}\) by writing
\[\left( \prod_{i\ne i_0}\theta_i \right) \psi_{i_0} \cdot \frac{\psi_l}{\psi_{i_0}}= \frac{\theta_l}{\theta_{i_0}}\cdot \left( \prod_{i\ne l}\theta_i \right)\psi_l,\]
and the dangling constant \(\frac{\theta_l}{\theta_{i_0}}\) is the value of the integral.

\vspace{0.5cm}

{\bf Acknowledgement.} X. Z. want to thank R. Peled for his encouragement of writing this paper. {A.C. was supported by ARC grant  DP180100613 and Australian Research Council Centre of Excellence for Mathematical and Statistical Frontiers (ACEMS) CE140100049. X.Z. was supported
  by ERC  starting grant 678520 and supported in part by Israel Science Foundation grant 861/15.}


\begin{thebibliography}{DMR17b}

\bibitem[ACK14]{angel2014localization}
Omer Angel, Nicholas Crawford, and Gady Kozma.
\newblock {Localization for linearly edge reinforced random walks}.
\newblock {\em Duke Mathematical Journal}, 163(5):889--921, 2014.

\bibitem[AM93]{aizenman1993localization}
Michael Aizenman and Stanislav Molchanov.
\newblock Localization at large disorder and at extreme energies: An elementary
  derivations.
\newblock {\em Communications in Mathematical Physics}, 157(2):245--278, 1993.

\bibitem[BHS18]{bauerschmidt2018dynkin}
Roland Bauerschmidt, Tyler Helmuth, and Andrew Swan.
\newblock Dynkin isomorphism and mermin--wagner theorems for hyperbolic sigma
  models and recurrence of the two-dimensional vertex-reinforced jump process.
\newblock {\em arXiv preprint arXiv:1802.02077}, 2018.

\bibitem[BS10]{articleBS}
Anne-Laure Basdevant and Arvind Singh.
\newblock Continuous-time vertex reinforced jump processes on galton-watson
  trees.
\newblock {\em Annals of Applied Probability}, 22, 05 2010.

\bibitem[CD87]{coppersmith1987random}
Don Coppersmith and Persi Diaconis.
\newblock {Random walk with reinforcement}.
\newblock {\em Unpublished manuscript}, pages 187--220, 1987.

\bibitem[Col09]{collevecchio2009}
Andrea Collevecchio.
\newblock Limit theorems for vertex-reinforced jump processes on regular trees.
\newblock {\em Electron. J. Probab.}, 14:1936--1962, 2009.

\bibitem[CZ18]{chen2018speed}
Xinxin Chen and Xiaolin Zeng.
\newblock Speed of vertex-reinforced jump process on galton--watson trees.
\newblock {\em Journal of Theoretical Probability}, 31(2):1166--1211, 2018.
\bibitem[DV02]{davis2002continuous}
Burgess Davis and Stanislav Volkov.
\newblock {Continuous time vertex-reinforced jump processes}.
\newblock {\em Probability theory and related fields}, 123(2):281--300, 2002.

\bibitem[DV04]{davis2004vertex}
Burgess Davis and Stanislav Volkov.
\newblock {Vertex-reinforced jump processes on trees and finite graphs}.
\newblock {\em Probability theory and related fields}, 128(1):42--62, 2004.

\bibitem[DR06]{diaconis2006bayesian}
Persi Diaconis and Silke~WW Rolles.
\newblock {Bayesian analysis for reversible Markov chains}.
\newblock {\em The Annals of Statistics}, pages 1270--1292, 2006.
\bibitem[DMR17a]{DMR15}
Margherita Disertori, Franz Merkl, and Silke W.~W. Rolles.
\newblock A supersymmetric approach to martingales related to the
  vertex-reinforced jump process.
\newblock {\em Lat. Am. J. Probab. Math. Stat}, 2017.

\bibitem[DMR17b]{disertori2017martingales}
Margherita Disertori, Franz Merkl, and Silke~WW Rolles.
\newblock Martingales and some generalizations arising from the supersymmetric
  hyperbolic sigma model.
\newblock {\em arXiv preprint arXiv:1710.02308}, 2017.



\bibitem[DS10]{DS10}
Margherita Disertori and Tom Spencer.
\newblock {Anderson localization for a supersymmetric sigma model}.
\newblock {\em Communications in Mathematical Physics}, 300(3):659--671, 2010.

\bibitem[DSZ10]{DSZ06}
Margherita Disertori, Tom Spencer, and Martin~R Zirnbauer.
\newblock {Quasi-diffusion in a 3D supersymmetric hyperbolic sigma model}.
\newblock {\em Communications in Mathematical Physics}, 300(2):435--486, 2010.


\bibitem[KR00]{keane2000edge}
Michael~S Keane and Silke~WW Rolles.
\newblock Edge-reinforced random walks on finite graphs.
\newblock {\em Verhandelingen KNAW}, 52, 2000.

\bibitem[LW17]{letac2017multivariate}
G{\'e}rard Letac and Jacek Weso{\l}owski.
\newblock Multivariate reciprocal inverse gaussian distributions from the
  sabot-tarr$\backslash$es-zeng integral.
\newblock {\em arXiv preprint arXiv:1709.04843}, 2017.

\bibitem[M{\"O}R08]{merkl2008magic}
Franz Merkl, Aniko {\"O}ry, and Silke~WW Rolles.
\newblock The ‘magic formula’for linearly edge-reinforced random walks.
\newblock {\em Statistica Neerlandica}, 62(3):345--363, 2008.

\bibitem[MRT16]{merkl2016convergence}
Franz Merkl, Silke~WW Rolles, and Pierre Tarr{\`e}s.
\newblock Convergence of vertex-reinforced jump processes to an extension of
  the supersymmetric hyperbolic nonlinear sigma model.
\newblock {\em arXiv preprint arXiv:1612.05409}, 2016.

\bibitem[RN18]{2018arXiv181006905R}
Olivier {Raimond} and Tuan-Minh {Nguyen}.
\newblock {Strongly vertex-reinforced jump process on a complete graph}.
\newblock {\em arXiv e-prints}, page arXiv:1810.06905, Oct 2018.

\bibitem[ST15a]{ST15}
Christophe Sabot and Pierre Tarr\`es.
\newblock {{Edge-reinforced random walk, vertex-reinforced jump process and the
  supersymmetric hyperbolic sigma model}}.
\newblock {\em J. Eur. Math. Soc.}, 17(9):2353--2378, 2015.

\bibitem[ST15b]{sabot2013ray}
Christophe Sabot and Pierre Tarres.
\newblock {{Inverting Ray-Knight identity}}.
\newblock {\em Prob. Th. Rel. Fields, online first}, 2015.

\bibitem[STZ17]{STZ15}
Christophe Sabot, Pierre Tarr{\`e}s, and Xiaolin Zeng.
\newblock The vertex reinforced jump process and a random schr{\"o}dinger
  operator on finite graphs.
\newblock {\em The Annals of Probability}, 45(6A):3967--3986, 2017.

\bibitem[SZ19]{SZ15}
Christophe Sabot and Xiaolin Zeng.
\newblock A random schr{\"o}dinger operator associated with the vertex
  reinforced jump process on infinite graphs.
\newblock {\em Journal of the American Mathematical Society}, 32(2):311--349,
  2019.

\bibitem[Zen16]{zeng2013vertex}
Xiaolin Zeng.
\newblock How vertex reinforced jump process arises naturally.
\newblock In {\em Annales de l'Institut Henri Poincar{\'e}, Probabilit{\'e}s et
  Statistiques}, volume~52, pages 1061--1075. Institut Henri Poincar{\'e},
  2016.

\bibitem[Zir91]{Zirnbauer91}
Martin~R Zirnbauer.
\newblock {Fourier analysis on a hyperbolic supermanifold with constant
  curvature}.
\newblock {\em Communications in mathematical physics}, 141(3):503--522, 1991.

\end{thebibliography}
\end{document}